\documentclass[11pt,reqno]{amsart} 

\usepackage{amsthm,amsfonts,amscd,amsmath}
\usepackage[hidelinks]{hyperref} 
\usepackage{url}
\usepackage{xcolor}
\hypersetup{
    colorlinks,
    linkcolor={red!50!black},
    citecolor={blue!50!black},
    urlcolor={blue!80!black}
}

\usepackage[normalem]{ulem}
\usepackage{graphicx} 
\usepackage{datetime} 
\usepackage{cancel}
\usepackage{caption,subcaption}
\numberwithin{figure}{section}
\numberwithin{table}{section}

\usepackage{stmaryrd} 
\usepackage{framed} 

\let\citep\cite

\usepackage{upgreek}

\renewcommand{\chi}{\upchi}

\DeclareMathOperator{\Id}{Id}

\DeclareMathOperator{\sgn}{sgn}

\DeclareMathOperator{\Sq}{Sq}

\newcommand{\One}{\mathbb{1}}

\renewcommand{\Re}{\operatorname{Re}}

\title[Magic partitions -- Sign smoothing convolutions]{
       Magic partition functions: 
       Sign smoothing convolutions with 
       Dirichlet invertible arithmetic functions
} 
\author[M. D. Schmidt]{
     Dr.\ Maxie Dion Schmidt \\ 
     \url{maxieds@gmail.com}
} 

\email{\url{maxieds@gmail.com}}

\date{\today} 


\allowdisplaybreaks 

\theoremstyle{plain} 
\newtheorem{theorem}{Theorem}

\numberwithin{theorem}{section}

\theoremstyle{definition} 

\newtheorem{remark}[theorem]{Remark}
\newtheorem{definition}[theorem]{Definition}

\begin{document} 

\begin{abstract}  
Sign changes in sums of arithmetic functions and their inverses are a subtle topic 
with room to grow new results. Suppose that $S_f(x) := \sum_{n \leq x} f(n)$ is the 
summatory function of some arithmetic function $f$ such that $f(1) \neq 1$. There are known 
lower bounds on the limiting growth of $V(S_f, Y)$ -- the number of sign changes of $S_f(y)$ 
on the interval $y \in (0, Y]$ as $Y \rightarrow \infty$. 
We observe a partition theoretic 
sign smoothing by discrete convolution of the local oscillatory properties of the 
Dirichlet inverse of $f$, $S_{f^{-1}}(x)$. 
These so-called invertible ``magic partition function`` encodings 
lead to a sequence of convolution sums which have predictable sign properties provided 
the sequence of $f(n)$ ($f^{-1}(n)$, respectively) 
has reasonable asymptotic upper bounds with respect to $n$. 
\end{abstract}

\maketitle

\section{Introduction} 

\subsection{Dirichlet convolutions and Dirichlet inverse functions} 

For any fixed $f$, we 
define its summatory function for all positive integers $x \geq 1$ by 
\[
S_f(x) := \sum_{n \leq x} f(n). 
\]
Given any two arithmetic functions $f$ and $g$, we define their 
\emph{Dirichlet convolution}, $f \ast g$, to be the divisor sum 
\[
(f \ast g)(n) := \sum_{d|n} f(d) g\left(\frac{n}{d}\right), \text{\ for\ all\ } n \geq 1. 
\] 
The multiplicative inverse with respect to Dirichlet convolution is defined by 
$\varepsilon(n) \equiv \delta_{n,1}$ so that $f \ast \varepsilon = \varepsilon \ast f = f$ for 
any arithmetic $f$. If $f(1) \neq 1$, then it is Dirichlet invertible. That is, there is another 
arithmetic function $f^{-1}(n)$ such that $f \ast f^{-1} = f^{-1} \ast f = \varepsilon$. 
Moreover, the function $f^{-1}$ is unique when it exists and satisfies the recursive formula 
\[
f^{-1}(n) = \begin{cases} 
     \frac{1}{f(1)}, & \text{if $n = 1$; } \\ 
     -\frac{1}{f(1)} \times \sum\limits_{\substack{d|n \\ d>1}} 
     f(d) f^{-1}\left(\frac{n}{d}\right), & \text{if $n \geq 2$. } 
     \end{cases} 
\]
Other formulas for $f^{-1}$ when $f$ is any Dirichlet invertible arithmetic function provide limited 
insight into distributions of the signs of the inverse function over $n \geq 1$. 
A partition theoretic motivation for expressing the Dirichlet inverse of any $f$ such that 
$f(1) \neq 0$ provides that for $n > 1$: 
\begin{equation} 
\label{eqn_fInvDirInv_PartitionsFormula} 
f^{-1}(n) = \sum_{k=1}^{\Omega(n)} (-1)^k \left\{ 
     \sum_{{\lambda_1+2\lambda_2+\cdots+k\lambda_k=n} \atop {\lambda_1, \lambda_2, \ldots, \lambda_k | n}} 
     \frac{(\lambda_1+\lambda_2+\cdots+\lambda_k)!}{1! 2! \cdots k!} 
     f(\lambda_1) f(\lambda_2)^2 \cdots f(\lambda_k)^k\right\}. 
\end{equation} 
Let the $m$-fold convolution of an arithmetic function $g$ with itself (i.e., convolve 
$g$ with itself $m$ times in a row at $n$) be denoted by $[g]_{\ast_m}$. 
Mousavi and Schmidt proved that \cite{MOUSAVI-SCHMIDT-2019} 
\begin{equation} 
\label{eqn_} 
f^{-1}(n) = \frac{\varepsilon(n)}{f(1)} + 
     \sum_{j=0}^{\left\lfloor \frac{\Omega(n)}{2} \right\rfloor} \left( 
     [f-f(1)\cdot\varepsilon]_{\ast_{2j+1}}(n) - f(1) \times 
     [f-f(1)\cdot\varepsilon]_{\ast_{2j}}(n)
     \right) \frac{1}{f(1)^{2j+1}}. 
\end{equation} 

\subsection{Some curious convolution experiments with partition functions}

We have experimentally observed an interesting trend of so-called 
\emph{sign smoothing transformations} of certain signed integer sequences under 
discrete convolution based encodings involving special partition functions. 

\begin{definition}
Let the next partition functions be defined for all $n \geq 0$ 
as the coefficients of the following product-based generating functions: 
\begin{subequations}
\begin{align} 
q(n) & := [q^n] \prod_{m \geq 1} (1+q^m) \\ 
\notag 
       & \phantom{:} = [q^n]\left(1+q+q^2+2q^3+2q^4+3q^5+4q^6+5q^7+6q^8+8q^9+10q^{10}+ 
       12q^{11}+\cdots\right) \\ 
q^{\ast}(n) & := [q^n] \prod_{m \geq 1} (1+q^m)^{-1} \\ 
\notag 
       & \phantom{:} = [q^n]\left(1-q-q^3+q^4-q^5+q^6-q^7+2q^8-2q^9 + 
       2q^{10}-2q^{11}+3q^{12} + \cdots \right) \\ 
p^{\ast}(n) & := [q^n] \prod_{m \geq 1} \left(1-q^m\right) \\ 
\notag
       & \phantom{:} = 
       [q^n]\left(1 - q - q^2 + q^5 + q^7 - q^{12} - q^{15} + \cdots \right) \\ 
p(n) & := [q^n] \prod_{m \geq 1} \left(1-q^m\right)^{-1} \\ 
\notag
       & \phantom{:} = 
       [q^n]\left(1 + q + 2q^2 + 3q^3 + 5q^4 + 7q^5 + 11q^6 + 15q^7 + 22q^8 + 
       30q^9 + 42q^{10} + \cdots \right).
\end{align} 
\end{subequations}
In words, $q(n)$ is the number of (orderless) partitions of $n$ into distinct parts, 
which is the same as the number of partitions of $n$ with odd parts, 
$q^{\ast}(n)$ is the Cauchy product inverse of $q(n)$, 
$p(n)$ is the number of (orderless) partitions of $n$, and 
$p^{\ast}(n)$ is the Cauchy product inverse of $p(n)$. 
\end{definition}

\begin{remark}[Asymptotics of the special partition functions]
\label{remark_SpecPartFnAsymptoticExps}
We can apply the 
circle method, or saddle point arguments, 
to the $q$-series generating functions of these sequences to obtain the following 
limiting asymptotics as $n \rightarrow \infty$
\cite{ANALYTIC-COMB,HARDY-RAMANUJAN-ASYMPFORMULAE,HEIKOTODT-THESIS} 
\cite[\S 26.10(v)]{NISTHB}: 
\begin{subequations}
\begin{align} 
\label{eqn_PFuncsP1P2n_asymptotics_stmt_v1} 
q(n) & = \frac{3^{3/4}}{4 \sqrt[4]{3} \cdot n^{\frac{3}{4}}} 
     \exp\left(\pi\sqrt{\frac{n}{3}}\right) \left( 
     1 + O\left(\frac{1}{\sqrt{n}}\right)
     \right) \\ 
q^{\ast}(n) & = \frac{(-1)^n}{2 \cdot 24^{1/4} n^{\frac{3}{4}}} 
     \exp\left(\pi\sqrt{\frac{n}{6}}\right) 
     \left(
     1 + O\left(\frac{1}{\sqrt{n}}\right)
     \right) \\ 
p(n) & = \frac{1}{4\sqrt{3} n} \exp\left(\pi \sqrt{\frac{2n}{3}}\right) \left( 
     1 + O\left(\frac{1}{\sqrt{n}}\right)
     \right).
\end{align} 
\end{subequations}
\end{remark}

\begin{definition}[Convolution-based encodings]
\label{def_TfEncodingsSpecialCasePFunc_examples_v1}
\label{def_Cvl_fPartFns}
We define two invertible transformations, or encodings, that are respective inverses of one another 
on any fixed arithmetic $f$ as the discrete 
convolution sums given by 
\begin{subequations}
\begin{align} 
\label{eqn_TfEncodingsSpecialCasePFunc_examples_v1} 
c_1[f](n) & := \sum_{1 \leq j \leq n} f(j) q(n-j) \\ 
c_2[f](n) & := \sum_{1 \leq j \leq n} f(j) q^{\ast}(n-j) \\ 
c_3[f](n) & := \sum_{1 \leq j \leq n} f(j) p^{\ast}(n-j) \\ 
c_4[f](n) & := \sum_{1 \leq j \leq n} f(j) p(n-j).
\end{align}
\end{subequations}
\end{definition}

Several tables of the convolutions in Definition \ref{def_Cvl_fPartFns} are shown in 
the tables tabulated in the appendix section 
(see Appendix \ref{Appendix_tables}) 
\citep[Appendix A]{MDS-GTTHESIS}. 
The sign-smoothing properties of the functions $f$ ($f^{-1}$, respectively) 
convolved with the special partition functions in 
Definition \ref{def_TfEncodingsSpecialCasePFunc_examples_v1}, we notice the 
following properties: 

\begin{theorem}[Sign smoothing convolutions]
\label{theorem_MagicPartitionsObs_v1}
Let the \emph{sign function}, $\sgn(h(n)) \mapsto \{\pm 1\}$, 
be defined to be $|h(n)| / h(n)$ whenever $h(n) \neq 0$ where 
$\sgn(h(n)) := 1$ if $h(n) = 0$. 
Furthermore, suppose that $f(n) \geq 1$ for all $n$ and that 
$f(n) \ll q^{\ast}(n)$ as $n \rightarrow \infty$. 
Then for all sufficiently large $n$, we have that 
\begin{subequations}
\begin{align}
\tag{A}
\sgn\left\{c_2[f^{-1}](n+1)\right\} & = - \sgn\left\{c_2[f^{-1}](n)\right\}.
\end{align}
We also have that 
\begin{equation}
\tag{B}
\sgn\left\{c_1[f^{-1}](n)\right\}, 
\end{equation}
\end{subequations}
is eventually constant for all sufficiently large $n$. 
\end{theorem}

Depending on the fixed sequence, $f(n)$, we sometimes have analogous properties of the 
convolution sequences, $c_3[f](n)$ and $c_4[f](n)$, though this is not always the case 
(see the tables in Appendix \ref{Appendix_tables}). 
The exponentially dominant asymptotic properties of each of the four special 
partition functions shown in 
Remark \ref{remark_SpecPartFnAsymptoticExps} 
suggest generalizations of 
Theorem \ref{theorem_MagicPartitionsObs_v1} 
as a topic for future work. 

\subsection{Local sign changes of an arithmetic function} 

The sign changes of an arithmetic function $f$ are often considered in applications 
where we must estimate the growth of sums depending on $f$ when $f$ (or especially $f^{-1}$) 
oscillates in sign. 
There are several existing results that characterize the expected number of sign changes of 
well enough behaved $f$ on increasingly large intervals of consecutive integers. 
Let $V(f, Y)$ denote the number of sign changes of $f$ on the interval $(0, Y]$ for 
real $Y > 0$: 
\[
V(f, Y) := \sup \left\{N: \exists \{x_i\}_{i=1}^N, 0<x_1<\cdots<x_N \leq Y, 
     f(x_i) \neq 0, \sgn(f(x_i)) \neq \sgn(f(x_{i+1})), \forall 1 \leq i < N\right\}. 
\]
It is known that the analytic properties, poles, and zeros of the 
Dirichlet generating function (DGF) of $f$ provide key insights into the 
sign changes of these functions \cite{OSCPROPS-ARITHFUNCSI}. 
For example, if the DGF of $f$ is analytic on some half-plane, subject to certain 
restrictions, then Landau showed in 1905 that 
the summatory function of $f$, $S_f(x)$, changes signs infinitely often as we let 
$x$ tend to infinity. We also have the next 
theorem that extends Landau's and which provides a more precise minimal 
statement concerning the frequency of the sign changes of $S_f(x)$. 

\begin{theorem}[P\'olya] 
Suppose that $S_f(x)$ is real-valued for all $x \geq x_0$, and define the function 
$\hat{F}_f(s)$ by the Mellin transform at $-s$ as 
\[
\hat{F}_f(s) := \int_{x_0}^{\infty} \frac{S_f(x)}{x^{s+1}} dx. 
\]
Suppose that $\hat{F}_f(s)$ is analytic for all $\Re(s) > \theta$, but is not analytic 
in any half-plane $\Re(s) > \theta - \varepsilon$ for $\varepsilon > 0$. 
Furthermore, suppose that $\hat{F}_f(s)$ is meromorphic in some half-plane 
$\Re(s) > \theta - c_0$ for some $c_0 > 0$. Let 
\[
\gamma_f := \begin{cases} 
     \inf \{|t|: \hat{F}_f(s) \text{\ is not analytic at\ } s = \theta+\imath t\}, & \text{\rm
     if $f$ is not analytic at $\Re(s) = \theta$; } \\ 
     \infty, & \text{\rm otherwise.}
     \end{cases}. 
 \]
Then 
 \[
 \limsup_{Y \rightarrow \infty} \left\{\frac{V(S_f, Y)}{\log Y}\right\} \geq \frac{\gamma_f}{\pi}. 
\]
\end{theorem} 

\section{Proof of the main theorem} 

\begin{proof}[Proof of Theorem \ref{theorem_MagicPartitionsObs_v1}(A)]
To simplify notation, we define the constants 
\[
A := \frac{1}{2 \cdot 24^{\frac{1}{4}}}, k := \frac{\pi}{\sqrt{6}}, 
\]
i.e., so that 
\[
q^{\ast}(n) = \frac{(-1)^n A \exp\left(k \sqrt{n}\right)}{n^{\frac{3}{4}}} \left( 
     1 + O\left(\frac{1}{\sqrt{n}}\right) 
     \right). 
\]
We then compute the following difference for $1 \leq j \leq n$: 
\begin{align*}
q^{\ast}(n + 1 - j) - q^{\ast}(n - j) & = (-1)^{n+1-j} A \left[ 
     \frac{\exp\left(k \sqrt{n+1-j}\right)}{(n+1-j)^{\frac{3}{4}}} - 
     \frac{\exp\left(k \sqrt{n-j}\right)}{(n-j)^{\frac{3}{4}}}
     \right] \left(1 + O\left(\frac{1}{\sqrt{n+1-j}}\right)\right) \\ 
     & = \frac{(-1)^{n+1-j} A}{(n+1-j)^{\frac{3}{4}}} \left[ 
     \exp\left(k \sqrt{n+1-j}\right) + 
     \frac{\exp\left(k \sqrt{n-j}\right)}{\left(1 - \frac{1}{n+1-j}\right)^{\frac{3}{4}}}
     \right] \left(1 + O\left(\frac{1}{\sqrt{n+1-j}}\right)\right) \\ 
     & = \left(\frac{2A (-1)^{n+1-j} \exp\left(k \sqrt{n+1-j}\right)}{ 
     (n+1-j)^{\frac{3}{4}}} + O\left(\frac{\exp\left(k \sqrt{n+1-j}\right)}{ 
     (n+1-j)^{\frac{9}{4}}}\right)\right) \times \\ 
     & \phantom{\qquad} \times 
     \left(1 + O\left(\frac{1}{\sqrt{n+1-j}}\right)\right) 
\end{align*} 
It follows that this difference of the encoding sums satisfies 
\[
c_2[f^{-1}](n + 1) - c_2[f^{-1}](n) = f^{-1}(n + 1) - 2 c_2[f^{-1}](n) + 
     o\left(c_2[f^{-1}](n)\right). 
\]
We assert that because of our assumption on the asymtotic growth of $f(n)$ as 
$n \rightarrow \infty$, we have that 
$f^{-1}(n + 1) = o\left(c_2[f^{-1}](n)\right)$ as $n \rightarrow \infty$. 
This completes the proof that (A) holds for all sufficiently large $n$. 
\end{proof}

A similar, but somewhat repetitive, argument shows that part (B) of the theorem 
holds when we convolve with the strictly positive partition function, $q(n)$. 


\section{Conclusions} 
\label{subSection_Concl_OpenQsGens}

We have several Dirichlet convolution identities of the form 
$\lambda^{-1} = \left(\One_{\Sq} \ast \mu\right)^{-1}$, 
$\varphi^{-1} = \left(\Id_1 \ast \mu\right)^{-1}$, and 
$\mu^{-1} = 1$ where 
$\One_{\Sq}$ is the characteristic function of the squares and where 
$\Id_1(n) = n$ for all $n \geq 1$ is the identity function. 
The summatory functions $L(x) := \sum_{n \leq x} \lambda(n)$ and 
$M(x) := \sum_{n \leq x} \mu(n)$ have signed, oscillatory summands whose partial sums are 
notoriously hard to estimate precisely -- especially within short intervals. 
This point also helps us to motivate why the particularly interesting cases of the 
tables in Appendix \ref{Appendix_tables} 
highlight the sign smoothing properties of the signed Dirichlet inverse functions, $f^{-1}$. 
Note that the discrete (Cauchy product, or generating function) convolutions prescribed by 
Definition \ref{def_TfEncodingsSpecialCasePFunc_examples_v1} 
are invertible so that any such partition function 
encoding over $f$ ($f^{-1}$, respectively) 
can be reversed to recover the terms of the original sequence. 

\subsection*{Acknowledgements} 

Thanks to Dr.\ Hamed Mousavi at the University of Bristol 
for suggesting generalizations to the observations on the magic partition function 
experiments from my doctoral thesis at Georgia Tech.

\bigskip\hrule\medskip
\appendix
\clearpage

\clearpage
\section{Tables of experiments}
\label{Appendix_tables} 

\begin{table}[ht!]
\small
\begin{tabular}{|l||l|l|l|l|l||l|l|l|l|l|} \hline
$n$ & $f(n)$ & $c_1[f](n)$ & $c_2[f](n)$ & $c_3[f](n)$ & $c_4[f](n)$ & $f^{-1}(n)$ & $c_1[f^{-1}](n)$ & $c_2[f^{-1}](n)$ & $c_3[f^{-1}](n)$ & $c_4[f^{-1}](n)$ \\ \hline
$1$ & $1$ & $1$ & $1$ & $1$ & $1$ & $1$ & $1$ & $1$ & $1$ & $1$ \\ \hline
$2$ & $1$ & $2$ & $0$ & $2$ & $2$ & $-1$ & $0$ & $-2$ & $0$ & $0$ \\ \hline
$3$ & $2$ & $4$ & $1$ & $4$ & $5$ & $-2$ & $-2$ & $-1$ & $-2$ & $-1$ \\ \hline
$4$ & $2$ & $7$ & $-1$ & $7$ & $9$ & $-1$ & $-2$ & $0$ & $-2$ & $-2$ \\ \hline
$5$ & $4$ & $12$ & $2$ & $12$ & $18$ & $-4$ & $-7$ & $-1$ & $-7$ & $-7$ \\ \hline
$6$ & $2$ & $17$ & $-4$ & $17$ & $28$ & $2$ & $-6$ & $6$ & $-6$ & $-8$ \\ \hline
$7$ & $6$ & $27$ & $4$ & $27$ & $50$ & $-6$ & $-13$ & $-7$ & $-13$ & $-21$ \\ \hline
$8$ & $4$ & $39$ & $-6$ & $39$ & $76$ & $-1$ & $-20$ & $8$ & $-20$ & $-30$ \\ \hline
$9$ & $6$ & $53$ & $5$ & $53$ & $121$ & $-2$ & $-23$ & $-5$ & $-23$ & $-51$ \\ \hline
$10$ & $4$ & $74$ & $-10$ & $74$ & $178$ & $4$ & $-31$ & $15$ & $-31$ & $-69$ \\ \hline
\end{tabular}
\caption{Sign smoothing transformations when $f(n) \equiv \varphi(n)$ is the Euler totient function}
\end{table}

\begin{table}[ht!]
\small
\begin{tabular}{|l||l|l|l|l|l||l|l|l|l|l|} \hline
$n$ & $f(n)$ & $c_1[f](n)$ & $c_2[f](n)$ & $c_3[f](n)$ & $c_4[f](n)$ & $f^{-1}(n)$ & $c_1[f^{-1}](n)$ & $c_2[f^{-1}](n)$ & $c_3[f^{-1}](n)$ & $c_4[f^{-1}](n)$ \\ \hline
$1$ & $1$ & $1$ & $1$ & $1$ & $1$ & $1$ & $1$ & $1$ & $1$ & $1$ \\ \hline
$2$ & $2$ & $3$ & $1$ & $3$ & $3$ & $-2$ & $-1$ & $-3$ & $-1$ & $-1$ \\ \hline
$3$ & $2$ & $5$ & $0$ & $5$ & $6$ & $-2$ & $-3$ & $0$ & $-3$ & $-2$ \\ \hline
$4$ & $3$ & $9$ & $0$ & $9$ & $12$ & $1$ & $-1$ & $2$ & $-1$ & $-2$ \\ \hline
$5$ & $2$ & $13$ & $-2$ & $13$ & $20$ & $-2$ & $-5$ & $0$ & $-5$ & $-6$ \\ \hline
$6$ & $4$ & $20$ & $1$ & $20$ & $35$ & $4$ & $-2$ & $5$ & $-2$ & $-5$ \\ \hline
$7$ & $2$ & $28$ & $-4$ & $28$ & $54$ & $-2$ & $-4$ & $-6$ & $-4$ & $-12$ \\ \hline
$8$ & $4$ & $39$ & $2$ & $39$ & $86$ & $0$ & $-9$ & $4$ & $-9$ & $-16$ \\ \hline
$9$ & $3$ & $54$ & $-4$ & $54$ & $128$ & $1$ & $-6$ & $-4$ & $-6$ & $-24$ \\ \hline
$10$ & $4$ & $71$ & $4$ & $71$ & $192$ & $4$ & $-7$ & $8$ & $-7$ & $-28$ \\ \hline
\end{tabular}
\caption{Sign smoothing transformations when $f(n) \equiv d(n)$ is the divisor function}
\end{table}

\begin{table}[ht!]
\small
\begin{tabular}{|l||l|l|l|l|l||l|l|l|l|l|} \hline
$n$ & $f(n)$ & $c_1[f](n)$ & $c_2[f](n)$ & $c_3[f](n)$ & $c_4[f](n)$ & $f^{-1}(n)$ & $c_1[f^{-1}](n)$ & $c_2[f^{-1}](n)$ & $c_3[f^{-1}](n)$ & $c_4[f^{-1}](n)$ \\ \hline
$1$ & $1$ & $1$ & $1$ & $1$ & $1$ & $1$ & $1$ & $1$ & $1$ & $1$ \\ \hline
$2$ & $2$ & $3$ & $1$ & $3$ & $3$ & $-2$ & $-1$ & $-3$ & $-1$ & $-1$ \\ \hline
$3$ & $2$ & $5$ & $0$ & $5$ & $6$ & $-2$ & $-3$ & $0$ & $-3$ & $-2$ \\ \hline
$4$ & $2$ & $8$ & $-1$ & $8$ & $11$ & $2$ & $0$ & $3$ & $0$ & $-1$ \\ \hline
$5$ & $2$ & $12$ & $-1$ & $12$ & $19$ & $-2$ & $-4$ & $-1$ & $-4$ & $-5$ \\ \hline
$6$ & $3$ & $18$ & $0$ & $18$ & $32$ & $5$ & $0$ & $6$ & $0$ & $-2$ \\ \hline
$7$ & $2$ & $25$ & $-2$ & $25$ & $50$ & $-2$ & $-1$ & $-8$ & $-1$ & $-8$ \\ \hline
$8$ & $2$ & $34$ & $-1$ & $34$ & $77$ & $-2$ & $-8$ & $3$ & $-8$ & $-11$ \\ \hline
$9$ & $2$ & $46$ & $-1$ & $46$ & $115$ & $2$ & $-2$ & $-3$ & $-2$ & $-15$ \\ \hline
$10$ & $3$ & $61$ & $2$ & $61$ & $170$ & $5$ & $-1$ & $10$ & $-1$ & $-14$ \\ \hline
\end{tabular}
\caption{Sign smoothing transformations when $f(n) \equiv \omega + 1$}
\end{table}

\begin{table}[ht!]
\small
\begin{tabular}{|l||l|l|l|l|l||l|l|l|l|l|} \hline
$n$ & $f(n)$ & $c_1[f](n)$ & $c_2[f](n)$ & $c_3[f](n)$ & $c_4[f](n)$ & $f^{-1}(n)$ & $c_1[f^{-1}](n)$ & $c_2[f^{-1}](n)$ & $c_3[f^{-1}](n)$ & $c_4[f^{-1}](n)$ \\ \hline
$1$ & $1$ & $1$ & $1$ & $1$ & $1$ & $1$ & $1$ & $1$ & $1$ & $1$ \\ \hline
$2$ & $3$ & $4$ & $2$ & $4$ & $4$ & $-3$ & $-2$ & $-4$ & $-2$ & $-2$ \\ \hline
$3$ & $15$ & $19$ & $12$ & $19$ & $20$ & $-15$ & $-17$ & $-12$ & $-17$ & $-16$ \\ \hline
$4$ & $105$ & $125$ & $89$ & $125$ & $129$ & $-96$ & $-112$ & $-82$ & $-112$ & $-114$ \\ \hline
$5$ & $945$ & $1073$ & $838$ & $1073$ & $1094$ & $-945$ & $-1060$ & $-845$ & $-1060$ & $-1075$ \\ \hline
$6$ & $10395$ & $11484$ & $9437$ & $11484$ & $11617$ & $-10305$ & $-11379$ & $-9349$ & $-11379$ & $-11495$ \\ \hline
\end{tabular}
\caption{Sign smoothing transformations when $f(n) \equiv n!!$ is the double factorial function. 
         Notice that the growth of this function is approximately 
         $n!! \sim C \cdot \left(\frac{n}{e}\right)^n$ with $C > 0$ a constant, 
         which is by far asymptotically dominant compared to the partition function asymptotics given in 
         Remark \ref{remark_SpecPartFnAsymptoticExps}. 
         Then we can compute that the convolution encoding properties from part (A) of Theorem \ref{theorem_MagicPartitionsObs_v1} 
         fail for this sequence of $f(n)$. }
\end{table}

\begin{table}[ht!]
\small
\begin{tabular}{|l||l|l|l|l|l||l|l|l|l|l|} \hline
$n$ & $f(n)$ & $c_1[f](n)$ & $c_2[f](n)$ & $c_3[f](n)$ & $c_4[f](n)$ & $f^{-1}(n)$ & $c_1[f^{-1}](n)$ & $c_2[f^{-1}](n)$ & $c_3[f^{-1}](n)$ & $c_4[f^{-1}](n)$ \\ \hline
$1$ & $-1$ & $-1$ & $-1$ & $-1$ & $-1$ & $-1$ & $-1$ & $-1$ & $-1$ & $-1$ \\ \hline
$2$ & $0$ & $-1$ & $1$ & $-1$ & $-1$ & $0$ & $-1$ & $1$ & $-1$ & $-1$ \\ \hline
$3$ & $-1$ & $-2$ & $-1$ & $-2$ & $-3$ & $1$ & $0$ & $1$ & $0$ & $-1$ \\ \hline
$4$ & $1$ & $-2$ & $3$ & $-2$ & $-3$ & $-1$ & $-2$ & $-1$ & $-2$ & $-3$ \\ \hline
$5$ & $-1$ & $-3$ & $-3$ & $-3$ & $-7$ & $1$ & $-1$ & $1$ & $-1$ & $-3$ \\ \hline
$6$ & $1$ & $-4$ & $4$ & $-4$ & $-8$ & $-1$ & $-2$ & $-2$ & $-2$ & $-6$ \\ \hline
$7$ & $-1$ & $-5$ & $-5$ & $-5$ & $-15$ & $1$ & $-3$ & $3$ & $-3$ & $-7$ \\ \hline
$8$ & $2$ & $-6$ & $7$ & $-6$ & $-17$ & $-2$ & $-4$ & $-5$ & $-4$ & $-13$ \\ \hline
$9$ & $-2$ & $-8$ & $-10$ & $-8$ & $-30$ & $1$ & $-5$ & $5$ & $-5$ & $-15$ \\ \hline
$10$ & $2$ & $-10$ & $11$ & $-10$ & $-35$ & $-2$ & $-7$ & $-6$ & $-7$ & $-26$ \\ \hline
\end{tabular}
\caption{Sign smoothing transformations when $f(n) \equiv q^{\ast}(n)$ 
         (to show challenges to growth of $f$ compared to the corresponding partition functions)}
\end{table}

\begin{table}[ht!]
\small
\begin{tabular}{|l||l|l|l|l|l||l|l|l|l|l|} \hline
$n$ & $f(n)$ & $c_1[f](n)$ & $c_2[f](n)$ & $c_3[f](n)$ & $c_4[f](n)$ & $f^{-1}(n)$ & $c_1[f^{-1}](n)$ & $c_2[f^{-1}](n)$ & $c_3[f^{-1}](n)$ & $c_4[f^{-1}](n)$ \\ \hline
$1$ & $1$ & $1$ & $1$ & $1$ & $1$ & $1$ & $1$ & $1$ & $1$ & $1$ \\ \hline
$2$ & $2$ & $3$ & $1$ & $3$ & $3$ & $-2$ & $-1$ & $-3$ & $-1$ & $-1$ \\ \hline
$3$ & $3$ & $6$ & $1$ & $6$ & $7$ & $-3$ & $-4$ & $-1$ & $-4$ & $-3$ \\ \hline
$4$ & $5$ & $12$ & $1$ & $12$ & $15$ & $-1$ & $-4$ & $1$ & $-4$ & $-5$ \\ \hline
$5$ & $7$ & $21$ & $1$ & $21$ & $29$ & $-7$ & $-13$ & $-3$ & $-13$ & $-15$ \\ \hline
$6$ & $11$ & $36$ & $2$ & $36$ & $54$ & $1$ & $-14$ & $8$ & $-14$ & $-20$ \\ \hline
$7$ & $15$ & $59$ & $1$ & $59$ & $95$ & $-15$ & $-31$ & $-15$ & $-31$ & $-49$ \\ \hline
$8$ & $22$ & $94$ & $3$ & $94$ & $163$ & $-10$ & $-52$ & $11$ & $-52$ & $-77$ \\ \hline
$9$ & $30$ & $146$ & $2$ & $146$ & $270$ & $-21$ & $-77$ & $-17$ & $-77$ & $-141$ \\ \hline
$10$ & $42$ & $222$ & $5$ & $222$ & $439$ & $-14$ & $-117$ & $26$ & $-117$ & $-214$ \\ \hline
\end{tabular}
\caption{Sign smoothing transformations when $f(n) \equiv p(n)$ is the ordinary partition function}
\end{table}

\end{document}